\documentclass{amsart}

\usepackage[utf8]{inputenc}

\usepackage{amsmath}
\usepackage{amsthm}
\usepackage{amssymb}
\usepackage{amsrefs}
\usepackage{hyperref}
\usepackage{aliascnt}
\usepackage{enumitem}
\usepackage{svg}
\usepackage{quiver}
\usepackage{accents}

\counterwithin{figure}{section}

\newaliascnt{theorem}{figure}
\newtheorem{theorem}[theorem]{Theorem}

\aliascntresetthe{theorem}

\newaliascnt{lemma}{figure}
\newtheorem{lemma}[lemma]{Lemma}

\aliascntresetthe{lemma}

\newaliascnt{proposition}{figure}
\newtheorem{proposition}[proposition]{Proposition}

\aliascntresetthe{proposition}

\newaliascnt{conjecture}{figure}

\aliascntresetthe{conjecture}

\newaliascnt{corollary}{figure}
\newtheorem{corollary}[corollary]{Corollary}

\aliascntresetthe{corollary}

\newaliascnt{question}{figure}

\aliascntresetthe{question}

\theoremstyle{definition}

\newaliascnt{definition}{figure}
\newtheorem{definition}[definition]{Definition}

\aliascntresetthe{definition}

\newaliascnt{algorithm}{figure}

\aliascntresetthe{algorithm}

\newaliascnt{problem}{figure}
\newtheorem{problem}[problem]{Problem}

\aliascntresetthe{problem}

\newaliascnt{example}{figure}

\aliascntresetthe{example}

\theoremstyle{remark}

\newaliascnt{remark}{figure}

\aliascntresetthe{remark}

\newaliascnt{observation}{figure}

\aliascntresetthe{observation}

\newaliascnt{notation}{figure}

\aliascntresetthe{notation}

\makeatletter
\def\itemautorefname~#1\null{(#1)\null}
\def\equationautorefname~#1\null{(#1)\null}
\makeatother
\include{angbr}

\NewDocumentCommand{\underhat}{m}{%
\underaccent{\hat}{#1}
}
\DeclareMathOperator{\Core}{Core}
\DeclareMathOperator{\rank}{rank}

\title{Two problems about bases and independent sets in free groups}
\author[]{Lucy Koch-Hyde}
\address{Department of Mathematics, CUNY Graduate Center, New York, NY 10016}
\email{Lucy.Koch-Hyde96@gc.cuny.edu}
\author[]{\'Eamonn Olive}
\address{Department of Mathematics, CUNY Graduate Center, New York, NY 10016}
\email{Eamon.Olive89@gc.cuny.edu}
\date{November 2026}

\begin{document}

\begin{abstract}
  We will answer a problem of Dotsenko that asks for a tight upper bound on the rank of subgroups given a maximum length for its generators.
  Then, we will use the same tools to answer a longstanding problem of Grigorchuk in the 10th edition of the Kourovka notebook.
  We accomplish this using Stallings foldings and simple graph theory.
\end{abstract}

\maketitle

\section{Introduction}

We begin with the following elementary definitions:
\begin{definition}
  $S_n \subset F_r$ is the set of all freely reduced words of length exactly $n$, and $B_n \subset F_r$ is the set of all freely reduced words of length at most $n$.
\end{definition}
These correspond to spheres and balls around the identity of $F_r$ under the word metric.

\begin{definition}
    Let $X$ be a subset of a free group $F_r$.
    The {\it growth function} of $X$, $g_X : \mathbb N \to \mathbb N$, is defined as:
    \[
      g_X(n) = \left|X\cap B_n\right|
    \]
\end{definition}
\begin{definition}[\cites{grigorchuk_problems_1997}]
    The {\it exponential growth rate} of $X\subseteq F_r$ is
    \[
      \limsup_{n\to \infty} \sqrt[n]{g_X(n)}
    \]
\end{definition}
Although the above definition is standard, an equivalent definition is helpful to indicate why this measure is notable:
\begin{definition}
    The {\it exponential growth rate} of $X\subseteq F_r$ is
    \[
      \inf\{b:g_X\in O(b^n)\}
    \]
\end{definition}
The exponential growth rate may also be called the ``growth exponent'' (for example, in \cite{grigorchuk_symmetrical_1980}).
For the history of growth and its relation to groups see \cite{grigorchuk_problems_1997}.
While there has been much attention given to the exponential growth rate of groups and their subgroups, in this paper we look instead at the growth rate of particular subsets of the free group.

In 1986 Grigorchuk posed the following question:
\begin{problem}[\cite{khukhro_unsolved_2026}*{10.13}]\label{prob:Kourovka}
    Does there exist a set $X\subset F_2$ such that:
    \begin{enumerate}
    \item\label{itm:normal independence} for all $x \in X$, $x\notin \left\llangle X - \{x\}\right\rrangle$.
    \item $X$ is ``massive''.
    That is, it has exponential growth rate of $3$ (the exponential growth rate of $F_2$ itself).
    \end{enumerate}
\end{problem}

This problem is relevant in the context of group presentations.
For the group presentation $\left\langle F_2 \mid X \right\rangle$ (giving the group $F_2/\left\llangle X\right\rrangle$), \autoref{itm:normal independence} imposes the restriction that no part of $X$ is redundant, which is a desirable property.
\autoref{prob:Kourovka} then asks a question about how ``big'' a non-redundant relator set can be in terms of its asymptotic growth.

We note that \autoref{itm:normal independence} is a more restrictive form of an ``independent set'':
\begin{definition}
    A subset $X$ of a group $G$ is {\it independent} when for all $x\in X$, $x\notin \left\langle X- \{x\}\right\rangle$.
\end{definition}

We will show that the answer to \autoref{prob:Kourovka} is ``No such $X$ exists''.
In fact, we will prove a more general result:

\begin{theorem}\label{thm:independent set limit}
    Let $S\subseteq B_n\subseteq F_r$ be an independent set.
    \[
    |S|\in O\left(\sqrt{2r-1}^n\right)
    \]
\end{theorem}

In the process of showing the above we will address a related open problem posed by Dotsenko:

\begin{problem}[\cite{kapovich_combinatorial_2026}*{F42}]\label{prob:Dotsenko}
    What is the highest rank of a subgroup of $F_r$ generated by elements of:
    \begin{enumerate}[label=(\alph*)]
        \item $S_n$?
        \item $B_n$?
    \end{enumerate}
\end{problem}

Our answer comes in the following theorem:
\begin{theorem}\label{thm:max rank} \hfill
  \begin{enumerate}[label=(\alph*)]
    \item The maximum possible rank of a subgroup of $F_r$ generated by elements of $S_n$ has two cases depending on the parity of $n$:
    \[
      \left\{
      \begin{array}{rl}
        \phantom{r}\sqrt{2r-1}^{n\phantom{-1}} & n \mathrm{\,\,is\,\,even} \\
        r\sqrt{2r-1}^{n-1} & n \mathrm{\,\,is\,\,odd}
      \end{array}
      \right.
    \]
    \item The greatest rank of a subgroup of $F_r$ generated by elements of $B_n$ is the same as for $S_n$ for all $n$.
  \end{enumerate}
\end{theorem}

\section{Stallings foldings and core graph}

A common approach to the combinatorics of subgroups of free groups derives from two facts:
\begin{enumerate}
    \item The free groups are exactly the fundamental groups of graphs.
    \item There is a Galois connection between subgroups of $\pi_1(X)$ and covering spaces of $X$.
\end{enumerate}

If we fix a graph to represent a free group, we can then capture a subgroup of the free group as a covering space of that graph.
Naturally we would like to make that graph as simple as possible, so for a free group $F_r$ we choose the ``rose graph'' $R_r$.
That is, the graph with 1 vertex and $r$ edges.
We can encode a covering as the covering graph with an edge labeling and edge direction to determine the quotient map.
Since this graph is a covering space, our labeling is a regular edge coloring.
That is, for any vertex there is exactly one outgoing edge of for each generator in $F_r$, and likewise there is exactly one type of each incoming edge.
This graph encoding a covering space of a rose graph is usually called a ``Schreier coset graph'' or simply a ``Schreier graph''.

It is possible for a subgroup of finite rank to give a Schreier graph with infinitely many vertices.
However, in these cases there are only finitely many vertices and edges which are ``essential''.
That is most of the edges and vertices lie on trees, and so are invisible to the fundamental group.
Thus, we want to perform a deformation retract onto a finite graph, removing all the trees.
This we call the ``Stallings core graph''.
This is no longer a covering space, but we still have an injection from the fundamental group into $F_r$.

Now we will provide a more formal definition of the Stallings core graph:

\begin{definition}
    Let $G$ be a subgroup of $F_r$.
    The \textit{Stallings core graph of $G$}, denoted $\Core(G)$, is triplet $(Q,p,L)$ where
    \begin{itemize}
      \item $Q$ is a quiver.
      \item $p$ is a vertex.
      \item $L$ is an edge coloring by generators in $F_r$.
    % That is, a function from edges to generators of $F_r$, such that at every vertex no pair of outgoing edges have the same color, and no pair of incoming edges have the same color.
    \end{itemize}
    with the following properties:
    \begin{itemize}
        \item
        The induced homomorphism $L^\ast : \pi_1(Q,p) \to F_r$ has image exactly equal to $G$.
        Since $L$ is an edge coloring, $L^\ast$ is injective.
        
        \item For every edge $e$, there is a homotopy class of curves $[\gamma]\in \pi_1(Q,p)$ such that every curve in $[\gamma]$ contains $e$.
        In other words, $Q$ is connected, and it is not the wedge product of graph with a tree.
        %In yet more words, $Q$ is the minimal graph with the desired properties.
    \end{itemize}
\end{definition}

A detailed introduction can be found in \cite{kapovich_stallings_2002}.
We also recommend \cite{office_hours} as an elementary introduction with nice pictures.

A crucial property of Stallings core graphs is the following:
\begin{lemma}[\cite{kapovich_combinatorial_2026}*{Lemma 8.2}]\label{lem:stallings rank}
  Let $G$ be a subgroup of $F_r$.
  Then the
  \begin{equation*}
    \operatorname{rank}(G)=E - V + 1
  \end{equation*}
  where $E$ and $V$ are the number of edges and vertices, respectively, in $\Core(G)$.
\end{lemma}
Hence, we can find bounds for the rank of a subgroup by finding bounds for the number of edges and vertices in its Stallings core graph.

\begin{lemma}\label{lem:vertex distance}
  Let $S \subseteq B_n$.
  The distance between the base point and any other vertex $\Core(\langle S \rangle)$ is at most $\left\lfloor \frac{n}{2} \right\rfloor$.
  Furthermore, there are no edges between vertices of distance $\frac{n}{2}$ from the base point.
\end{lemma}
\begin{proof}
  Let us consider the process of finding $\Core(\langle S \rangle)$ through Stallings foldings.
  Initially we start with cycles corresponding to words in $S$ joined together at the base point.
  At this point, since each cycle has a circumference of at most $n$, all vertices are within $\left\lfloor \frac{n}{2} \right\rfloor$ of the base point.
  Furthermore, for a vertex to be $\frac{n}{2}$ from the base point, $n$ must be even and the vertex can only have edges to vertices of distance $\frac{n}{2} - 1$ from the base point.
  To obtain the core graph, we would apply a sequence of foldings, each of which are graph quotients, and so the distance between any two vertices will not increase.
  Thus, the bound on distance still holds in $\Core(\langle S \rangle)$.
  Graph quotients also cannot result in new edges between vertices.
  This means that any vertex in $\Core(S)$ with an edge to a vertex of distance $\frac{n}{2}$ from the base point must have in its preimage a vertex of distance $\frac{n}{2} - 1$ from the base point, and so it can be no further than $\frac{n}{2} - 1$ from the base point.
\end{proof}

\begin{lemma}\label{lem:odd vertex count}
  If $S \subseteq B_{2k+1} \subset F_r$, then $\Core(\langle S \rangle)$ contains at most $\frac{r(2r-1)^k - 1}{r-1}$ vertices.
\end{lemma}
\begin{proof}
  Since $S \subset F_r$, $\Core(\langle S\rangle)$ has a maximum degree of $2r$, and so it follows that the base point can have at most $2r$ neighbors, each of which can have at most $2r-1$ new neighbors, and so on.
  Since $\Core(\langle S \rangle)$ cannot contain any vertex further than $\left\lfloor \frac{2k+1}{2} \right\rfloor = k$ from the base point, and so there can be at most
  \begin{equation*}
    1 + 2r\sum_{d=1}^k (2r-1)^{d-1} = \frac{r(2r-1)^k - 1}{r-1}
  \end{equation*}
  vertices in $\Core(\langle S \rangle)$.
\end{proof}
\begin{corollary}\label{cor:odd edge bound}
  If $S \subseteq B_{2k+1} \subset F_r$, then $\Core(\langle S \rangle)$ contains at most $r\frac{r(2r-1)^k - 1}{r-1}$ edges.
\end{corollary}
\begin{proof}
  The max degree of a Stallings core graph for a subgroup of $F_r$ is at most $2r$.
  Since each edge is incident to two (not necessarily distinct) vertices, there can be no more than $r$ times as many edges as vertices in a Stallings core graph.
  Hence, for a subset $S \subseteq B_{2k+1} \subset F_r$, $\Core(\langle S \rangle)$ contains at most $r\frac{r(2r-1)^k - 1}{r - 1}$ edges.
\end{proof}

\begin{lemma}\label{lem:even edge bound}
  If $S \subseteq B_{2k}$, then $\Core(\langle S \rangle)$ contains at most $r\frac{(2r-1)^k - 1}{r-1}$ edges.
\end{lemma}
\begin{proof}
  Because $\Core(\langle S \rangle)$ has a maximum degree of $2r$, the base point has at most $2r$ edges coming out of it, each of which connects to a vertex that has at most $2r - 1$ edges that have not already been counted.
  We can continue counting the new edges on vertices of each distance from the base point like this to get an upper bound.
  However, it is important to note that by \autoref{lem:vertex distance}, vertices of distance $k$ from the base point cannot have any edges that do not connect to a closer edge, so we get the following as our upper bound:
  \begin{equation*}
    2r\sum_{d = 0}^{k-1}(2r-1)^d = r\frac{(2r-1)^k - 1}{r-1}
  \end{equation*}
\end{proof}
\begin{corollary}\label{cor:even vertex count}
  If $S \subseteq B_{2k}$, then $\Core(\langle S \rangle)$ contains at most $\frac{r(2r-1)^k - 1}{r-1}$ vertices.
\end{corollary}
\begin{proof}
  As a Stallings core graph, $\Core(\langle S\rangle)$ must be connected, and so there can be at most one more vertex than there are edges.
\end{proof}
% We will combine \autoref{cor:odd edge bound} and \autoref{lem:even edge bound} into one statement for easy reference:
% \begin{proposition}\label{cor:vertex bound}
%     %If $S\subseteq B_{n}\subset F_r$, the the number of vertices in $\Core(\langle S\rangle)$ is bounded above by $\frac{\sqrt{2r-1}^{n+2}}{r-1}\in O\left(\sqrt{2r-1}^n\right)$.
%     If $S\subseteq B_{n}\subset F_r$, the the number of edges in $\Core(\langle S\rangle)$ is at most
%     \[
%     \begin{cases}r\frac{\sqrt{2r-1}^n-1}{r-1} & n \mathrm{\,\,is\,\, even}\\r\frac{r\sqrt{2r-1}^{n-1}-1}{r-1}& n \mathrm{\,\,is\,\, odd}\end{cases}\in O\left(\sqrt{2r-1}^n\right)
%     \]
% \end{proposition}
% 
\section{Maximum rank for bases}
With the results of the previous section, it is straightforward to prove an upper bound for \autoref{prob:Dotsenko}:
\begin{lemma}\label{lem:rank upper bound}
  The rank of a subgroup generated by elements of $B_n$ is at most
  \begin{equation*}
      \left\{
      \begin{array}{rl}
        \phantom{r}\sqrt{2r-1}^{n\phantom{-1}} & n \mathrm{\,\,is\,\,even} \\
        r\sqrt{2r-1}^{n-1} & n \mathrm{\,\,is\,\,odd}
      \end{array}
      \right.
  \end{equation*}
\end{lemma}
\begin{proof}
    Consider $G\subseteq B_n\subset F_r$.
    Let us say $\Core(G)$ has $E$ edges and $V$ vertices.
    We can establish a lower bound on the number of vertices as follows:
    Each edge is incident on two vertices, and each vertex can be incident on at most $2r$ edges.
    Thus,
    \[
        V \geq \frac{2E}{2r} = \frac{E}{r}
    \]
    Applying \autoref{lem:stallings rank} we can get a bound on the rank of the subgroup in terms of $E$:
    \begin{align*}
        E+1-\rank(G) &\geq \frac Er \\
        \frac{(r-1)E}{r}+1 &\geq \rank(G) 
    \end{align*}
    Now, supposing $n$ is even we can apply \autoref{lem:even edge bound} to get a bound on $\rank(G)$ in terms of $n$ and $r$:
    \begin{align*}
        \rank(G)
        &\leq \frac{(r-1)E}{r}+1 \\
        &\leq \frac{(r-1)r\frac{\sqrt{2r-1}^n-1}{r-1}}{r}+1 \\
        &=\sqrt{2r-1}^n
    \end{align*}
    Now we will obtain a bound on the rank of $G$ in terms of the vertices.
    \begin{align*}
        V &\geq \frac{E}{r} \\
        V &\geq \frac{\rank(G)+V-1}{r} \\
        (r-1)V+1 &\geq \rank(G)
    \end{align*}
    Now we apply \autoref{lem:odd vertex count} to get a bound in terms of $n$ and $r$
    \begin{align*}
        \rank(G)
        &\leq  (r-1)V+1 \\
        &\leq  (r-1)\frac{r\sqrt{2r-1}^{n-1}-1}{r-1}+1 \\
        &=  r\sqrt{2r-1}^{n-1}
    \end{align*}
\end{proof}
However, it remains to be seen that this upper bound is in fact attainable.
Here we will provide two classes of bases in $S_n$ with rank equal to the upper bound given by \autoref{lem:rank upper bound}.
Let $P_r(k) \subset F_r$ be the set of all words of the form $wx_i\iota(w)^{-1}$ where $x_i$ is a generator of $F_r$, $w$ is a word of length $k$ that does not end in $x_i^{-1}$, and $\iota$ is the automorphism of $F_r$ that maps each of the generators of $F_r$ to its inverse.
Fix $x_1$ to be a generator of $F_r$ and $\varphi$ to be an automorphism of $F_r$ that cyclically permutes the set of generators and inverses of generators of $F_r$, and let $Q_r(k) \subset F_r$ be the set of all words of the form $x_1w\varphi^{(m)}(x_1w)^{-1}$, where $w$ is a word of length $k-1$ that does not begin with $x_1^{-1}$ and $1 \leq m < 2r$.
To demonstrate that these are in fact bases we will show that they are Nielsen reduced sets using the conditions outlined in the following lemma:

\begin{lemma}[\cite{mks}*{Lemma 3.1}]
  \label{lem:isolated}
  Let $R \subseteq F_r$.
  Then $R$ is Nielsen reduced if and only if the following conditions are satisfied.
  \begin{enumerate}
    \item\label{itm:major segments} For each $w \in R$, its prefix and suffix of length $\left\lceil \frac{|w|}{2} \right\rceil$, its major initial and major terminal segment, do not appear as the prefix or suffix, respectively, of any other word in $R^{\pm 1}$.
    \item\label{itm:half} For each $w \in R$ of even length, its prefix and suffix of length $\frac{w}{2}$, its left and right half, do not appear as the prefix or suffix, respectively, of any other word in $R^{\pm 1}$.
  \end{enumerate}
\end{lemma}
We can now show that the sets given by $P_r(k)$ and $Q_r(k)$ are indeed bases with rank equal to the upper bound given in \autoref{lem:rank upper bound}, thereby solving \autoref{prob:Dotsenko}.

\begin{proof}[Proof of \autoref{thm:max rank}]
  Since $P_r(k)$ contains only words of length $2k + 1$, we only need to satisfy \autoref{itm:major segments} from \autoref{lem:isolated} to show that it is Nielsen reduced.
  No word in $P_r(k)$ can share a major initial or major terminal segment with a word in $P_r(k)^{-1}$, since the middle letter of any element in $P_r(k)$ is always positive, where as in $P_r(k)^{-1}$ it would always be negative.
  Furthermore, by examining the first or last $k + 1$ letters we can fully identify the values of $w$ and $x_i$, thereby uniquely identifying the entire word.
  For $Q_r(k)$, since all the words it contains are of length $2k$, we must show that both that conditions of \autoref{lem:isolated} are satisfied for all words in $Q_r(k)$.
  Because the right half of the word is a suffix of the major terminal segment, it is sufficient to show that for any word $x_1w\varphi^{(m)}(x_1w)^{-1} \in Q_r(k)$, neither its major initial segment nor its right half are the prefix or suffix, respectively, of any other word in $Q_r(k)^{\pm 1}$.
  The last letter of any word in $Q_r(k)$ is of the form $\varphi^{(m)}(x_1^{-1})$, and so, is not $x^{-1}$, which means that no word in $Q_r(k)$ share a non-trivial prefix or suffix with a word in $Q_r(k)^{-1}$.
  Furthermore, by examining either the major initial segment or the right half of a word in $Q_r(k)$, we can find $w$ and $m$, thereby uniquely identifying the word.

  Finally, let us quickly confirm that for $n = 2k + 1$, $|P_r(k)| = r\sqrt{2r-1}^{n-1}$ and for $n = 2k$, $|Q_r(k)| = \sqrt{2r-1}^n$.
  For $P_r(k)$, all words are of the form $w x_i \iota(w)^{-1}$, and so we have $r$ choices for $x_i$, and then for each selection of $x_i$, $(2r-1)^k$ choices for $w$, giving a total size of $r(2r-1)^k = r\sqrt{2r-1}^{n-1}$ words where $n = 2k + 1$.
  For $Q_r(k)$, all words are of the form $x_1 w \varphi^{(m)}(x_1 w)^{-1}$, which means we have $(2r-1)^{k-1}$ choices for $w$ and $2r-1$ choices for $m$, for a total of $(2r-1)^k = \sqrt{2r-1}^n$ words where $n = 2k$.

  Thus, the maximum possible rank of a subgroup generated by elements of $S_n$ is at least
  \begin{equation*}
    \left\{
    \begin{array}{rl}
      \phantom{r}\sqrt{2r-1}^{n\phantom{-1}} & n \mathrm{\,\,is\,\,even} \\
      r\sqrt{2r-1}^{n-1} & n \mathrm{\,\,is\,\,odd}
    \end{array}
    \right.
  \end{equation*}
  Because $S_n \subseteq B_n$, \autoref{lem:rank upper bound} shows that this is precisely the maximum possible rank in the $S_n$ case.
  This result applies for the $B_n$ case as well due to \autoref{lem:rank upper bound} and the fact that $S_n \subseteq B_n$.
\end{proof}

\begin{figure}[p]
\includesvg[width=.8\textwidth]{2-5.svg}
\caption{The Stallings core graph for one of the greatest rank subgroups of $F_2$ generated by words of length 5.}
\end{figure}
\begin{figure}[p]
\includesvg[width=.9\textwidth]{2-6.svg}
\caption{The Stallings core graph for one of the greatest rank subgroups of $F_2$ generated by words of length 6.}
\end{figure}

\section{Problem 10.13}

The results above can be applied to show \autoref{thm:independent set limit} and thus solve \autoref{prob:Kourovka}.

% This is more general in a few ways:
% \begin{itemize}
%     \item It is not restricted to free groups only of rank 2.
%     \item Each element of $S$ is only required to be outside of the group generated by its complement rather than its normal closure.
%     Since $\left\langle S-s\right\rangle\subseteq \left\llangle S-s\right\rrangle$, this is a weaker condition on $S$.
%     \item We provide a explicit upper bound on $g_S$.
% \end{itemize}

We start with a small technical lemma.

\begin{lemma}\label{lem:monotone special case}
  Let $K$ be a basis.
  Let $w$ be such that $\{w\}\cup K$ is an independent set.
  If $K\cup w$ is not a basis, then $\Core(\langle K, w\rangle)$ is a non-trivial quotient of $\Core(\langle K\rangle)$.
\end{lemma}
\begin{proof}
  Let $\hat w$ indicate the longest prefix of $w$ that can be followed as a path (originating from the base point) in $\Core(\langle K\rangle)$.
  Let $\underhat{w}$ indicate the longest prefix of $w^{-1}$ which can be followed in $\Core(\langle K\rangle)$.
  If $\left|\hat w\right|+\left|\underhat{w}\right| < |w|$, then there is an edge lying on $w$ which is not identified with an edge in $\Core(\langle K\rangle)$.
  Thus $\Core(\langle K,w\rangle)$ would be $\Core(\langle K\rangle)$ with an added ``bridge'' consisting of the edges which are in neither $\hat w$ or $\underhat{w}$.
  This bridge increases the rank by one and so $K\cup \{w\}$ is a basis.

  Thus if $K\cup \{w\}$ is not a basis, $\left|\hat w\right|+\left|\underhat{w}\right| \geq |w|$.
  In this case each edge of $w$ is identified with an edge of $\Core(\langle K\rangle)$, thus $\Core(\langle K,w\rangle)$ is a quotient.
  If it were a trivial quotient, then $\langle K\rangle =\langle K,w\rangle$ which would mean $w$ is not independent of $K$.
  Since we assumed $w$ was independent of $K$ it is a non-trivial quotient.
\end{proof}

This can be generalized further to get a lemma we are interested in using.

\begin{lemma}\label{lem:monotone}
    Let $S\subset F_r$ be a finite independent set.
    Let $K\subseteq S$ be a maximal basis.
    That is, $K$ is a basis of $\left\langle K\right\rangle$ and $K$ is not contained within a larger subset of $S$ forming a basis.
    Now let $w$ be such that $\{w\}\cup S$ is an independent set.
    If $K\cup w$ is not a basis, then $\Core(\langle S, w\rangle)$ is a non-trivial quotient of $\Core(\langle S\rangle)$.
\end{lemma}
\begin{proof}
    Our proof will proceed by induction on the size of $S-K$.
    Our base case of $S=K$ is exactly \autoref{lem:monotone special case}.

    Now assume via our inductive hypothesis that $\Core(\langle S\rangle)$ is a quotient of $\Core(\langle K\rangle)$.
    Let $f : \Core(\langle K\rangle) \twoheadrightarrow \Core(\langle K,w\rangle)$ be the quotient map provided by \autoref{lem:monotone special case}, and let $g : \Core(\langle K\rangle) \twoheadrightarrow \Core(\langle S\rangle)$ be the quotient map provided by our inductive hypothesis.

    Then $\Core(\langle S,w\rangle)$ is the pushout of $f$ and $g$.
    Thus there is a quotient map $\Core(\langle S\rangle) \twoheadrightarrow \Core(\langle S,w\rangle)$.%
    \begin{figure}
        \centering
        \begin{tikzcd}
            {\Core(\langle K\rangle)} & {\Core(\langle S\rangle)} \\
            {\Core(\langle K,w\rangle)} & {\Core(\langle S,w\rangle)}
            \arrow["g",two heads, from=1-1, to=1-2]
            \arrow["f"',two heads, from=1-1, to=2-1]
            %\arrow[dashed, two heads, from=1-1, to=2-2]
            \arrow[dashed, two heads, from=1-2, to=2-2]
            \arrow[dashed, two heads, from=2-1, to=2-2]
        \end{tikzcd}
        \caption{The commutative diagram demonstrating the pushout of $f$ and $g$}
        \label{fig:commutative diagram}
    \end{figure}
\end{proof}

\begin{corollary}\label{cor:nesting}
    Let $S\subset F_r$ be a finite independent set.
    Let $K\subseteq S$ be a maximal basis.
    Consider an ordering $K=S_0\subsetneq S_1\subsetneq S_2\subsetneq \cdots \subsetneq S_n=S$.
    If $i\leq j$ then $\Core(\langle S_j\rangle)$ has strictly fewer vertices than $\Core(\langle S_i\rangle)$.
\end{corollary}

We will now give a version of \autoref{thm:independent set limit} with more explicit bounds.
This is intended for ease of reference, but we think that this bound is likely weak.
\begin{theorem}
    Let $S\subseteq B_n\subseteq F_r$ be an independent set.
    \[
    |S|\in O\left(\sqrt{2r-1}^n\right)
    \]
    Specifically, $|S|$ can never be greater than the number of edges in the core graph of a subgroup generated by elements of $B_n$:
    \[
    |S|\leq \begin{cases}r\frac{\sqrt{2r-1}^n-1}{r-1} & n \mathrm{\,\,is\,\, even}\\r\frac{r\sqrt{2r-1}^{n-1}-1}{r-1}& n \mathrm{\,\,is\,\, odd}\end{cases}
    \]
\end{theorem}

\begin{proof}
    Let $K\subseteq S$ be a maximal basis.
    Now consider an ordering $s_1,s_2,\dots, s_n$ on $S-K$.
    Let $S_i = K\cup \{s_1,\dots,s_i\}$.
    By \autoref{cor:nesting} $\Core(\langle S_{i+1}\rangle)$ must have fewer vertices than $\Core(\langle S_i\rangle)$.
    Thus, $S-K$ must have less elements than $\Core(\langle K \rangle)$ has vertices.
    Thus, $|S|\leq \rank(\langle K\rangle) + V - 1$ where $V$ is the number of vertices in $\Core(\langle K\rangle)$.
    By \autoref{lem:stallings rank}, $\rank(\langle K\rangle) + V - 1$ is the number of edges in $\Core(\langle K\rangle)$, which we shall call $E$.
    Now we have that $|S|\leq E$, and by \autoref{lem:even edge bound} and \autoref{cor:odd edge bound}, we have an explicit bound on the number of edges.
\end{proof}

Now we will make explicit the application of this theorem to \autoref{prob:Kourovka}.

\begin{proof}[Solution to \autoref{prob:Kourovka}]
Let $S\subset F_r$ be an independent set.
$S\cap B_n$ is also an independent set for all $n$.
By definition, $g_S(n)=|S\cap B_n|$.
Thus, by \autoref{thm:independent set limit}, $g_S(n)\in O\left(\sqrt{2r-1}^n\right)$.
Therefore, the answer to \autoref{prob:Kourovka} is ``No''.
\end{proof}

An immediate question is whether the bound given in \autoref{thm:independent set limit} can be brought down.
While we maintain that we believe the fine bound can be brought down, we will show that one cannot bring the bound on the growth of independent sets much further down.

\begin{proposition}\label{prop:indpendent set lower bound}
    Let $F_r$ be a finitely generated non-Abelian free group.
    For every $b < \sqrt{2r-1}$, there exists an independent set $S\subset F_r$ such that
    \[
       g_S(n)\in \Omega\left(b^n\right)
    \]
\end{proposition}
\begin{proof}
    Let $\{x_1,x_2,\dots,x_r\}$ be a basis of $F_r$.
    Let $Y_n$ be all words in $F_r$ not containing $x_1^{\pm n}$ and not ending with $x_1^{\pm 1}$.
    It is clear that for every $c < 2r-1$ there is some $n$ such that
    \[
       g_{Y_n}(n)\in \Omega(c^n)
    \]
    That is as $n$ grows, the growth rates of $Y_n$ get closer and closer to that of $F_r$.

    Now let $n$ be the value given above when $c=b^2$.
    We will construct our independent set $S$ as all conjugates of $x_1^{2n}$ by elements of $Y_n$.
    It is clear that $S$ is an independent set.
    Furthermore, two elements $y_0,y_1\in Y_n$ give unique conjugates of $x_1^{2n}$ of lengths $2|y_0|+2n$ and $2|y_1|+2n$ respectively.
    Thus, $g_S(n)\in \Omega(\sqrt{c}^n)=\Omega(b^n)$.
\end{proof}

\section{Questions}

Since we have solved two problems, we would like to pose three.

We've shown a reasonably tight bound on the worst-case growth rate of an independent set (\autoref{thm:independent set limit} and \autoref{prop:indpendent set lower bound}), but it leaves a loose end:
% \begin{definition}
%     For a set $S\subseteq F_r$, define its {\it exponential growth rate} as
%     \[
%         \inf \{b\mid g_S(n)\in O(b^n)\}
%     \]
% \end{definition}
\begin{problem}
    In a finitely generated free group $F_r$, does there exist some independent subset with exponential growth rate exactly $\sqrt{2r-1}$?
\end{problem}

The solution to \hyperref[prob:Kourovka]{Grigorchuk's problem} is only the beginning of a broader problem.
Once we have shown no such set is massive, it is natural to ask what further can be shown about the growth of such a set.
Thus, we pose the following question as a next step:

\begin{problem}\label{prob:subexponential growth of normally independent sets}
    Does there exist a set $S\subset F_r$ such that:
    \begin{itemize}
    \item for all $s \in S$, $s\notin \left\llangle S - \{s\}\right\rrangle$.
    \item $S$ grows exponentially.
    That is, $b^n \in O(g_S(n))$ for some $b > 1$.
    \end{itemize}
\end{problem}

We conjecture the answer to \autoref{prob:subexponential growth of normally independent sets} is ``No''.

Lastly, we will pose a related question to \autoref{prob:Dotsenko}:
\begin{problem}
   What is the largest independent subset of:
   \begin{enumerate}[label=(\alph*)]
       \item $S_n\subset F_r$?
       \item $B_n\subset F_r$?
   \end{enumerate}
\end{problem}
Here we know that the answer lies in $\Theta\left(\sqrt{2r-1}^n\right)$ but a finer characterization would be interesting.

\bibliography{refs}

\end{document}